\documentclass[12pt]{amsart}
\usepackage{graphics,color}
\usepackage{amssymb,amsmath, latexsym, }
\usepackage{amssymb}
\usepackage{amsmath}
\usepackage{amsthm}
\usepackage{amsfonts}
\usepackage{mathrsfs}
\usepackage{listings}
\usepackage{color}
\usepackage[left]{lineno}
\usepackage{blindtext}
\usepackage[all,cmtip]{xy}

\makeatletter
\def\proof{\@ifnextchar[{\@oproof}{\@nproof}}
\def\@oproof[#1][#2]{\trivlist\item[\hskip\labelsep
\textit{#2 Completion of the proof of\ #1.}~]\ignorespaces}
\def\@nproof{\trivlist\item[\hskip\labelsep\textit{Proof.}~]\ignorespaces}

\makeatother

\setbox0=\hbox{$+$}
\newdimen\plusheight
\plusheight=\ht0
\def\+{\;\lower\plusheight\hbox{$+$}\;}

\setbox0=\hbox{$-$}
\newdimen\minusheight
\minusheight=\ht0
\def\-{\;\lower\minusheight\hbox{$-$}\;}

\setbox0=\hbox{$\cdots$}
\newdimen\cdotsheight
\cdotsheight=\plusheight%\ht0
\def\cds{\lower\cdotsheight\hbox{$\cdots$}}

\DeclareMathOperator{\rank}{rank}
\DeclareMathOperator{\Hom}{Hom}
\DeclareMathOperator{\divi}{div}
\DeclareMathOperator{\Sym}{Sym}
\DeclareMathOperator{\Gr}{Gr}
\DeclareMathOperator{\rk}{rk}

\DeclareMathOperator{\Supp}{Supp}
\DeclareMathOperator{\Max}{Max}

\numberwithin{equation}{section}
\newtheorem{theorem}{Theorem}[section]
\newtheorem{lemma}[theorem]{Lemma}
\newtheorem{cor}[theorem]{Corollary}
\newtheorem{proposition}[theorem]{Proposition}
\newtheorem{remark}[theorem]{Remark}

\newtheorem{definition}[theorem]{Definition}
\newtheorem{ex}[theorem]{Example}

\title{Asymptotic slopes and strong semistability  on surfaces}
\author[]{Mitra Koley}
\address{Stat-Math Unit, Indian Statistical Institute, 203 B.T. Road, Kolkata, India 700035}
\email{mitra.koley@gmail.com}
\author[]{A. J. Parameswaran}
\address{School of Mathematics, Tata Institute of Fundamental Research, Homi Bhabha Road, Mumbai, India 400005}
\email{param@math.tifr.res.in}

\begin{document}

\maketitle

%\linenumbers

\begin{abstract}
In this article we study asymptotic slopes of strongly semistable vector bundles on a smooth projective surface. A connection between asymptotic slopes and strong restriction theorem of a strongly semistable vector bundle is shown. We also  give an equivalent criterion of strong semistability of a vector bundle in terms of its asymptotic slopes under some assumptions on the surface and on the bundle.
\end{abstract}
\bigskip

\section{Introduction}
Let $X$ be a smooth/normal projective variety over an algebraically closed field $\mathbb{K}$ and $H$ be an ample line bundle on $X$. Let $E$ be a vector bundle on $X$. A subbundle $0\neq F\subset E$ of rank $k$, is said to be \emph{maximal}, if $\deg F$ is maximal among all subbundles of rank $k$. Maximal subbundles of vector bundles over a smooth projective curve have been studied by many authors. Maximal line subbundles  of a rank two bundle on a smooth projective curve have been studied in \cite{LN}. For higher rank vector bundles  again on curves, maximal subbundles are studied in \cite{MS} and  in many subsequent papers. 

In  \cite{PS}, the second author and Subramanian studied the behavior of maximal subbundles of a vector bundle on a smooth projective curve after finite pull backs.
%Prior to this in \cite{BP}, a criterion of strong semistability of vector bundles on curve is given in terms of slope of the maximal subbundle of a given rank 
We  briefly discuss their results  here.
Let $C$ be a smooth projective curve defined over an algebraically closed field $\mathbb{K}$ of arbitrary characteristic. Let $E$ be a vector bundle of rank $r$ over $C$. For each $1 \leq k < r$, the slope of maximal subbundle is denoted by 
$e_k(E)$, 
$$e_k(E):= \Max \{\frac{\deg(W)}{k}~| ~W \subset E \textrm{ is a subbundle of rank }k \}.$$

Define the asymptotic $k$-spectrum $\mathcal{AS}_k(E)$ and the asymptotic $k$-slope $\nu_k(E)$ as follows:

$$\mathcal{AS}_k(E):= \{\frac{e_k(f^* (E))}{\deg f}\}.$$
$$\nu_k(E ):= \textrm{ Limsup } \frac{e_k(f^* (E))}{\deg f}= \textrm{ Limsup }  \mathcal{AS}_k(E).$$

where the supremum is taken over all finite morphisms $f: D\to  C$.
One of their main result is the following:
\begin{theorem}[\protect{\cite[Theorem 4.1]{PS}}]
\label{Main-result-curve}
 Let $C$ be a smooth  projective curve defined over an algebraically closed field $\mathbb{K}$ of arbitrary characteristic and $E$ be a  vector bundle on $C$. Then $E$ is strongly semistable if and only if $\nu_k (E ) = \mu(E)$ for some $k$. Moreover if 
 $\nu_k (E ) = \mu(E)$ for some $k$,  then $\nu_j(E)=\mu(E)$ for all $j$.
\end{theorem}
Moreover in that paper \cite{PS}, the authors gave an explicit formula of $\nu_k(E)$ for an arbitrary vector bundle $E$ in terms of degrees and ranks of the strong Harder-Narasimhan factors of $E$.

Here in this article we study asymptotic $k$-spectrum of
vector bundles defined over a smooth projective surface and wanted to understand whether similar results still hold for smooth projective surfaces. First we show that
even if one defines asymptotic $k$-slope of a vector bundle similar to curve case;  one can not expect an analogue of Theorem~\ref{Main-result-curve}. This is because such an analogues result for rank $2$ vector bundles yields strong restriction theorem (Theorem~\ref{strong-restriction}). However we show that an analogues theorem for discriminant zero strongly semistable  bundles with some more additional conditions on the bundle and on the underlying surface (see section $4$ for more details). 

The paper is organized as follows: in Section 2, we recollect relevant definitions and some useful facts about strong semistability and discriminants.
In Section $3$ we prove that  an analogue theorem to Theorem~\ref{Main-result-curve}  for rank $2$ vector bundles implies strong restriction theorem (Theorem~\ref{strong-restriction}). Section $4$
is devoted on the study of asymptotic $k$-spectrum of vector bundles of arbitrary rank  with zero discriminants.% We also gave an explicit formula for asymptotic $k$-slope for some special kind of vector bundles.

\section{Preliminaries}
In this section we  recall some definitions and some useful facts  of  relevant topics which we need in later sections. Here  in general the underlying space is always assumed to be a normal projective surface.

Let $X$ be a smooth/normal projective surface over an algebraically closed field $\mathbb{K}$. Let $H$ be an ample line bundle on $X$ and $E$ be a torsion free sheaf of rank $r$ defined over $X$. Then the slope of $E$ with respect to $H$ is defined by
$$\mu(E)=\frac{c_1(E)\cdot H}{r}.$$
A torsion free sheaf $E$  on $X$ is called \emph{semistable}(resp. \emph{stable}) if for every nonzero subsheaf $W$ of $E$, $\mu(W)\leq \mu(E)$ (resp. $\mu(W)<\mu(E)$); equivalently  for every torsion free quotient  sheaf $Q$ of $E$, we have $\mu(Q)\geq \mu(E)$ (resp. $\mu(Q)>\mu (E)$). 

When $\mathbb{K}$ is a field of characteristic $p>0$,  let $F_X:X\to X$ denote the absolute Frobenius morphism. Then a vector bundle $E$ over $X$ is called \emph{strongly semistable}(resp. \emph{strongly stable}) if  for all $n\geq 0$, the Frobenius pull back ${F^n_X}^*(E)$ of $E$ is semistable (resp. stable).

Given a torsion free sheaf $E$, there exists a unique increasing filtration  of torsion free sheaves (known as the \emph{Harder-Narasimhan filtration})

$$E_{\bullet}: \quad \{0=F_0\subset F_1\subset \cdots \subset F_l=E\}$$

such that for each $i$, $F_i/F_{i-1}$ is a semistable torsion free sheaf with slope $\mu_i$ satisfying $\mu_i=\mu(\frac{F_i}{F_{i-1}})>\mu_{i+1}=\mu(\frac{F_{i+1}}{F_i})$.

The torsion free sheaves $\{F_i /F_{i-1}: 1\leq i\leq l\}$ are called \emph{the Harder-Narasimhan factors} of the bundle $E$. The factor  $F_1$ is called the \emph{maximal destabilizing subsheaf} of $E$. It's slope $\mu(F_1)$ is  denoted by
$\mu_{max}(E)$. The factor $E/F_{l-1}$ is called the \emph{minimal destabilizing quotient} of $E$ and it's  slope $\mu(E/F_{l-1})$ is denoted by $\mu_{min}(E)$.

If $\mathbb{K}$ is a field of characteristic $p>0$, then a filtration  
$$ E_{\bullet}: \quad \{0=F_0\subset F_1\subset \cdots \subset F_l=E\}$$
is called the \emph{strong Harder-Narasimhan} filtration of $E$, if 
it is the Harder-Narasimhan filtration of $E$ and for each $i$,
the factor $F_i/F_{i-1}$ is a strongly semistable sheaf. When $X$ is smooth,
by  a theorem of Langer (\cite{Langer}), for any torsion free sheaf $E$, there exists  an $n_0\in \mathbb{N}$ such that for all  ${F_X^{n_0}}^{*} (E)$ has  strong Harder-Narasimhan filtration.
Now we recall the following useful lemma.
\begin{lemma} (\cite[Lemma 1.3.3]{Lehn})
 Let  $F$ and $G$ be torsion free sheaves  such that  $\mu_{min}(F)>\mu_{max}(G)$  then $\Hom(F, G) = 0$.
\end{lemma}

Discriminant of a sheaf is an important invariant.  Let $E$ be a
torsion free sheaf of rank $r$  then the  discriminant $\Delta(E)$ of $E$  is defined by  $$\Delta(E)= (r -1)c_1(E)^2 - 2rc_2 (E),$$ where  for each $i$, $c_i(E)$ denotes the $i$-th Chern class of $E$.  In
particular, if $E$ is a vector bundle of rank $2$, then  $\Delta(E) = c_1(E)^2 -4c_2 (E)$.  In the following Proposition we write down a few facts about discriminant which we will use in the later sections. Let $X$ be a smooth surface.
\begin{proposition}
\label{discreminant-properties}
\begin{enumerate}
 \item Let $E$ be a torsion free sheaf  and $L$ be a line bundle  on $X$, then
       $$\Delta(E)=\Delta(E^*) \textrm{ and } \Delta(E\otimes L)=\Delta(E).$$
 \item  If $E$ is a strongly semistable torsion free sheaf  on $X$, then  $\Delta(E)\leq 0$. 
 %\item If $E$ is strongly semistable vector bundle over a surface defined over a  field of prime characteristic, then $\Delta(E)\leq 0$.  
 
 \item Let $E$ be a vector bundle  of rank $r$.  If $\Delta(E)=0$, 
 then $\Delta(S^n(E))=0$, for all $n>0$.
 \end{enumerate}
 \end{proposition}
 \begin{proof} $(1)$ is an easy computation and left to the reader.\\
 $(2)$ When $X$ is defined over a field of characteristic $0$, then it follows from  Theorem $3.4.1$ of \cite{Lehn}. When $X$ is defined over a field of  prime characteristic $p>0$, the proposition follows from Theorem $0.1$ of \cite{Langer}.\\
 $(3)$ This part may be known to experts. Since we are unable to find a reference,  we include a proof for the convenience of reader. \\
 By Lemma $10.1$ of \cite{Bogomolov},
    $$c_1(S^n(E))=\binom{n+r-1}{r} c_1(E)=P_r(n)c_1(E).$$
     $$c_2(S^n(E))=P_{r+1}(n)[c_2(E)-\frac{r-1}{2r}c_1(E)^2]+\frac{1}{2}[P^2_r(n)-\frac{n}{r}P_r(n)]c_1(E)^2,$$
     where $P_r(n)=\binom{n+r-1}{r}.$
      That is $$c_2(S^n(E))=-\frac{P_{r+1}(n)\Delta(E)}{2r}+\frac{1}{2}[P^2_r(n)-\frac{n}{r}P_r(n)]c_1(E)^2.$$
    Since $\Delta(E)=0$, in our case
    $$c_2(S^n(E))=\frac{1}{2}[P^2_r(n)-\frac{n}{r}P_r(n)]c_1(E)^2.$$
 Note that $\rank(S^n(E))=\binom{n+r-1}{r-1}$.
 \begin{eqnarray*}
  \Delta(S^n (E))&=&(\rk S^n(E)-1)c_1^2(S^n(E))-2\rk(S^n(E))c_2(S^n(E))\\
  &=&[\binom{n+r-1}{r-1}-1]{\binom{n+r-1}{r}}^2 c_1^2(E)-\\
  && \binom{n+r-1}{r-1}[{\binom{n+r-1}{r}}^2-\frac{n}{r}{\binom{n+r-1}{r}}]c_1^2(E)\\
  &=&[\frac{n}{r}{\binom{n+r-1}{r}}\binom{n+r-1}{r-1}-{\binom{n+r-1}{r}}^2]c_1^2(E)\\
  &=& 0.\\
 \end{eqnarray*}
\end{proof}

\begin{lemma}
\label{discriminant}
 Let $X$ be a smooth surface and $H$ be an ample line bundle on $X$.
 Let $0\to V\to E\to Q\to 0$ be an exact sequence of torsion free sheaves such that $\mu(V)=\mu(Q)$. If $E$ is strongly semistable  with $\Delta(E)=0$, then $\Delta(Q)=\Delta(V)=0$.
\end{lemma}

\begin{proof}
 Let $\rank E=n$.
 By definition $\Delta(E)=(n-1)c_1(E)^2-2nc_2(E)$.
 Let $\rank V=k$ and $\rank Q=l$. Then $n=k+l$.
 Then $\Delta(E)=(k+l-1)c_1(E)^2- 2(k+l)c_2(E)$. Therefore we have,
 \begin{eqnarray*}
 kl\Delta(E)&=&kl(k+l-1)c_1(E)^2- 2kl(k+l)c_2(E)\\
            &=&kl(k+l-1)(c_1(V)+c_1(Q))^2- 2kl(k+l)(c_1(V)c_1(Q)-c_2(V)-c_2(Q))\\
            &=& (k+l)l[(k-1)c_1(V)-2kc_2(V)]+(k+l)k[(l-1)c_1(Q)-2lc_2(Q)]\\
            & & +[l^2c_1(V)^2+k^2c_1(Q)^2-2klc_1(V)c_1(Q)]\\
            &=& (k+l)l\Delta(V)+(k+l)k\Delta(Q)+(lc_1(V)-kc_1(Q))^2.\\
 \end{eqnarray*}
 Hence
 \begin{eqnarray*}
 \frac{\Delta(E)}{k+l} &=& \frac{\Delta(V)}{k}+\frac{\Delta(Q)}{l}
                      +\frac{1}{kl(k+l)}(lc_1(V)-kc_1(Q))^2.\\
                      &\leq& \frac{\Delta(V)}{k}+\frac{\Delta(Q)}{l}
                      +\frac{kl}{(k+l)H^2}(\mu(V)-\mu(Q))^2\\
                      &=& \frac{\Delta(V)}{k}+\frac{\Delta(Q)}{l}.
 \end{eqnarray*}
 
 The middle inequality follows from Hodge index theorem and the last equality follows because $\mu(V)=\mu(Q)$.
 Since $E$ is strongly semistable and $\mu(V)=\mu(Q)=\mu(E)$, it follows that
 $V,Q$ are also strongly semistable. Hence  $\Delta(V)\leq 0$ and $\Delta(Q)
 \leq 0$ by Proposition~\ref{discreminant-properties} $(2)$. Since $\Delta (E)=0$, $\Delta(Q)=0$ and $\Delta(V)=0$.
 \end{proof}
 
 Now in next proposition we will see that any torsion free strong semistable sheaf with zero discriminant is a vector bundle.
 This fact might be known to experts but we include its prof for the sake of completeness.
 
 \begin{proposition}
 \label{delta-zero-implies-bundles}
 Let $(X,H)$ be a smooth polarized surface. Let $E$ be a torsion free  strongly semistable sheaf with $\Delta(E)=0$. Then $E$ is a vector bundle.
\end{proposition}
\begin{proof}
 Let $E^{**}$ denote the reflexive closure of $E$.
 Since $E$ is strongly semistable,  so is $E^{**}$.
 Consider the exact sequence
 $$0\to E\to E^{**}\to E^{**}/E\to 0.$$
 Since $E$ is torsion free, it is locally free in codimension $\geq 2$. Hence $\Supp(E^{**}/E)$ is a finite set of points. 
 Therefore $$c_1(E^{**}/E)=0 \quad \textrm{and}\quad c_2(E^{**}/E)\leq 0.$$ 
Note that $c_2(E^{**}/E)=0$ if and only if $\Supp(E^{**}/E)$ is empty, this is the case precisely when $E=E^{**}$.
Hence 
$$c_1(E^{**})=c_1(E) \quad \textrm{and}\quad c_2(E^{**})=c_2(E)+c_2(E^{**}/E).$$ 
Now \begin{eqnarray*}
     \Delta(E^{**})&=&(r-1)c_1^2(E^{**})-2rc_2(E^{**})\\
                   &=&(r-1)c_1^2(E)-2rc_2(E)-2rc_2(E^{**}/E)\\
                   &=&\Delta(E)-2rc_2(E^{**}/E).\\
    \end{eqnarray*}
Now since $E^{**}$ is strongly semistable, $\Delta(E^{**})\leq 0$.
Hence $\Delta(E)=0$ implies $c_2(E^{**}/E)=0$, i.e. $E=E^{**}$. 
Hence $E$ is reflexive. Now the proposition follows from the fact that reflexive sheaves on smooth surfaces are vector bundles.
\end{proof}

Next we observe that Theorem~$3.1$ of \cite{Langer2} which is proved stable bundles can be extended for the semistable case also. 
\begin{proposition}
\label{generalisation-of-Langer}
 Let $(X,H)$ be a smooth polarized surface with $H$ an ample line bundle. Let $E$ be a vector bundle of rank $r\geq 2$ with $\Delta(E)=0$. Assume that $E$ is strongly semistable. Let $C\in |H|$ be any smooth effective divisor, then $E|_C$ is also strongly semistable.
\end{proposition}

\begin{proof}
Let $E$ be a strongly semistable bundle with $\Delta(E)=0$. If $E$ is strongly stable we are done by  Theorem~$3.1$ of \cite{Langer2}. If not, there exists $e\geq 0$ and an exact sequence $0\to V\to {F_X^e}^*E\to Q\to 0 $,  such that $\mu(V)=\mu({F_X^e}^* E)=\mu(Q)$. Hence $V$  and $Q$ are also strongly semistable with $\Delta (V)=\Delta(Q)=0$ by Lemma~\ref{discriminant}. One also notes that by Proposition~\ref{delta-zero-implies-bundles} $V$ and $Q$ are also bundles. Since rank of $V$  and $Q$ are smaller than rank of $E$, by induction  on rank, $V|_C$  and $Q|_C$  are strongly semistable for all smooth effective divisor $C\in |H|$. Since $V,Q$ are bundles, the following is an exact sequence  of bundles. 
 $$0\to V|_C\to {F_X^e}^*E|_C\to Q|_C\to 0. $$
Now one can see that strong semistablity of $V|_C$  and $Q|_C$ implies ${F_C^e}^*E|_C$ is strongly semistable since $\mu(V|_C)=\mu({F_X^e}^*E|_C)=\mu(Q|_C)$ and hence $E|_C$ is  strongly semistable
for all smooth effective divisor $C\in |H|$.
\end{proof}

% By Jordan H\"{o}lder thereom there exists a filtration of torsion free stable sheaves with  same slopes as $E$.
% By Lemma~\ref{discriminant}, each factors of that filtrations are  of discreminant zero and hence each factor is a strongly  stable bundle with zero discriminant. Thus  by Theorem~$3.1$ of \cite{Langer2} for any smooth $C\in |H|$,  each factor restricted to $C$ is  strongly stable with equal slope. Hence the  proposition follows.

\section{Asymptotic slopes and strong restriction}

In this section we show that one can not expect an analogue of Theorem~\ref{Main-result-curve} for arbitrary semistable vector bundles on smooth projective surfaces. Recall that Theorem ~\ref{Main-result-curve} states that, a vector bundle $E$ on a smooth projective curve $C$ is  strongly semistable if and only if $\nu_K(E)=\mu(E)$ for all $k$. 

%Acually the proof of it shows more, it shows that if $E$ is strongly semistable vector bundle, then there exists a sequence of finite, separable covering $f_n:Y_n\to X$ and subbundles $W_n$ of $f_n^*E$ such that $\lim\limits_{n\to \infty}\frac{\mu(W_n)}{\deg f_n}=\mu(E)$. 

%%Let $E$ be a strongly semistable vector bundle on a  smooth polarized surface $X$ over algebraically closed field of characteristic $p>0$.There are atleast two possible generalizations of asymptotic spectrum, one is by considering maximal subbundles of $f^*E$, where $f:\tilde{X}\to X$ is a finite morphism with $\tilde{X}$ normal and another is by considering maximal subsheafs of $f^*E$, where $f:\tilde{X}\to X$ is a finite morphism with $\tilde{X}$ normal. Here in this section we deal with the first generalization and in the next section we deal with the second one.
%%\begin{definition}
%%Let $E$ be a vector bundle of rank $r$ over $X$. For each $1 \leq k < r$, the slope of maximal subbundle is denoted by  $e^b_k(E)$, 
%%$$e^b_k(E):= \Max \{\frac{\deg(W)}{k}~| ~W \subset E \textrm{ is a subbundle of rank }k \}.$$

%%Define the asymptotic $k$-spectrum $\mathcal{AS}^b_k(E)$ and the asymptotic $k$-slope $\nu^b_k(E)$ as follows:

%%$$\mathcal{AS}^b_k(E):= \{\frac{e_k^b(f^* (E))}{\deg f}\}.$$
%%$$\nu^b_k(E ):= \textrm{ Limsup } \frac{e^b_k(f^* (E))}{\deg f}= \textrm{ Limsup }  \mathcal{AS}^b_k(E).$$

%%where the supremum is taken over all finite morphisms $f: \tilde{X}\to X$ with $\tilde{X}$ normal.
%%\end{definition}

First we define asymptotic spectrum on surfaces.
Let $X$ be a smooth projective surface over an algebraically closed field $\mathbb{K}$ of characteristic $p>0$. Let $H$ be an ample line bundle on $X$.

\begin{definition}
 Let $E$ be a vector bundle of rank $r\geq 2$ on $X$.

For each $1 \leq k < r$, we denote the slope of maximal subsheaf of rank $k$ by $e_k(E)$, and 
$$e_k(E):= \Max \{\frac{\deg(W)}{k}~| ~W \subset E \textrm{ is a subsheaf of rank }k \}$$

Define the asymptotic $k$-spectrum $\mathcal{AS}_k(E)$ and the asymptotic $k$-slope $\nu_k(E)$ as follows:
Let $f:\tilde{X}\to X$ be a  finite  morphism with $\tilde{X}$ normal.
$$\mathcal{AS}_k(E):= \{\frac{e_k(f^* (E))}{\deg f}\}$$
$$\nu_k(E ):= \textrm{ Limsup } \frac{e_k(f^* (E))}{\deg f}= \textrm{ Limsup } \mathcal{AS}_k(E).$$

where the supremum is taken over all finite  morphisms $f: \tilde{X} \to  X$ with normal $\tilde{X}$.
\end{definition}

Consider $X=\mathbb{P}^2$ and $E=T\mathbb{P}^2$, the tangent bundle of $\mathbb{P}^2$.  It is known that $T\mathbb{P}^2$ is a strongly semistable bundle. With respect to the very ample line bundle $\mathcal{O}_X(1)$, $\mu(T\mathbb{P}^2)=3/2$.
In the following example we see that if we consider only composite of Frobenius morphisms $F^n:\mathbb{P}^2\to\mathbb{P}^2$,  and
the sequence $\{\frac{e_1({F^n}^* (E))}{\deg F^n}: n>0\}$, then 
$\textrm{ Limsup } \frac{e_1({F^n}^* (E))}{\deg F^n}< 3/2=\mu(T\mathbb{P}^2)$.

\begin{ex}
\label{TP^2}
 Consider $X=\mathbb{P}^2$ defined over a field of characteristic $p>0$. Let $E=T\mathbb{P}^2$, the tangent bundle of $\mathbb{P}^2$.  On a line $l\simeq \mathbb{P}^1$,  it is known that 
 $T\mathbb{P}^2|_l\simeq  \mathcal{O}_{\mathbb{P}^1}(2)\oplus \mathcal{O}_{\mathbb{P}^1}(1)$. 
 Hence $F^* T\mathbb{P}^2|_l\simeq \mathcal{O}_{\mathbb{P}^1}(2p)\oplus \mathcal{O}_{\mathbb{P}^1}(p)$. Consider the following exact sequence of sheaves:
 \begin{equation*}
 0\to K\to F^* T\mathbb{P}^2\to Q\to 0.
 \end{equation*}
 Now restricts this to any line $l\simeq \mathbb{P}^1$,  and using $F^* T\mathbb{P}^2|_l\simeq \mathcal{O}_{\mathbb{P}^1}(2p)\oplus \mathcal{O}_{\mathbb{P}^1}(p)$, we have
 \begin{equation}
  \label{*}
  0\to K|_{\mathbb{P}^1}\to \mathcal{O}_{\mathbb{P}^1}(2p)\oplus \mathcal{O}_{\mathbb{P}^1}(p) \to Q|_{\mathbb{P}^1}\to 0.
 \end{equation}
 
 If the map  $\mathcal{O}_{\mathbb{P}^1}(2p)\to Q|_{\mathbb{P}^1}$ induced from \eqref{*}
 is a zero map, then  the map
 $\mathcal{O}_{\mathbb{P}^1}(p) \to Q|_{\mathbb{P}^1}$  induced from \eqref{*} is surjective, hence $\mathcal{O}_{\mathbb{P}^1}(p) \simeq Q|_{\mathbb{P}^1}$. Therefore 
 $\mu(Q|_{\mathbb{P}^1})=p< 3p/2 =\mu(F^* T\mathbb{P}^2)|_{\mathbb{P}^1}$, which contradicts that $F^*T\mathbb{P}^2$ is a semistable  bundle.
 
 Hence the induced map $\mathcal{O}_{\mathbb{P}^1}(2p)\to Q|_{\mathbb{P}^1}$  from \eqref{*}, is nonzero, hence it is an injective map of sheaves, hence $\mu(Q)|_{\mathbb{P}^1}\geq 2p$. Hence  with respect to $\mathcal{O}_X(p), \mu(Q)\geq 2p^2$. Similar calculation for $F^n: \mathbb{P}^2\to\mathbb{P}^2$ shows that if $Q$ is any quotient of ${F^{n}}^* T\mathbb{P}^2$, then  with respect to $\mathcal{O}_X(p^n), \mu(Q)\geq 2p^{2n}> \mu({F^n}^* T\mathbb{P}^2)$. Hence $\textrm{ Limsup } \frac{e_1({F^n}^* (E))}{\deg F^n}< 3/2=\mu(T\mathbb{P}^2)$.
\end{ex}

Before stating our main result of the section we  prove a useful lemma.
\begin{lemma}
\label{lemma-strong-restriction}
 Let $E$ be a strongly semistable vector bundle on a smooth polarized surface $(X,H)$ over an algebraically closed field of characteristic $p>0$. For any $k$, if $\nu_k(E)=\mu(E)$, then $\nu_k(F_X^*E)=\mu(F_X^*E)$.
\end{lemma}

\begin{proof}
 Since $\nu_k(E)=\mu(E)$, there is a sequence of finite coverings $f_n:X_n\to X$ and a subbundles $F_n$ of $f_n^*E$ rank $k$ such that 
 $\frac{\mu( F_n)}{\deg f_n}\to \mu(E)$. Hence $\frac{\mu(F_{X_n}^* F_n)}{\deg f_n}\to \mu(F_X^*E)$. Since $F_X^*E$ is also strongly semistable, then $\nu_k(F_X^*E)\leq \mu(F_X^*E)$. By definition, 
 $\lim\limits_{n\to\infty}\frac{\mu(F_{X_n}^* F_n)}{\deg f_n}\leq \nu_k(F_X^*E
 )$. Hence the lemma.  
\end{proof}

%\begin{theorem}
%\label{strong-restriction}
% Let $(X,H)$ be a smooth polarized surface over an algebraically closed field of characteristic $p>0$ with $H$ very ample line bundle. Suppose $E$ be a rank $2$ strongly semistable vector bundle.  If there exists a sequence of finite covering 
%$f_n:Y_n\to X$ and line subbundles $W_n$ of $f_n^*E$ such that 
%$\lim\limits_{n\to \infty}\frac{\mu(W_n)}{\deg f_n}=\mu(E)$, then 
% $E|_C$ is strongly semistable for a general smooth  $C\in |H|$.
%\end{theorem}
Now we are ready to prove our main result of this section.
\begin{theorem}
\label{strong-restriction}
 Let $(X,H)$ be a smooth polarized surface over an algebraically closed field of characteristic $p>0$ with $H$ very ample line bundle. Suppose $E$ be a rank $2$ strongly semistable vector bundle.  If $\nu_1(E)=\mu(E)$, then 
 $E|_C$ is strongly semistable for a general smooth  $C\in |H|$.
\end{theorem}

\begin{proof}
Let $X$ be a smooth surface with a fixed  very ample polarization $H$. Let $E$ be a rank $2$ strongly semistable vector bundle on 
 $X$. 
 Let $C\in |H|$ be a general smooth curve such that ${F_C^{n}}^{*}E|_{C}$ is not semistable for some $n\geq 0$.  If $n>1$, then replacing $E$ by ${F_X^{n-1}}^*(E)$, and using Lemma~\ref{lemma-strong-restriction} we can assume that 
 $F_C^*E|_{C}$ is not semistable.
 Let
 $$0\to F_1\to F_C^{*}E|_{C}\to F_2\to 0$$
 be the Harder-Narasimhan filtration of $F_C^{*}E|_{C}$.
 
 Choose $0< \epsilon < \mu(F_1)-\mu(F_C^{*}E|_{C})$. Since $\nu_1(E)=\mu(E)$, there exists an $\tilde{X}\stackrel{f}{\to} X$ such that  $\tilde{X}$ is normal, $f$ is finite and an exact sequence 
  $$0\to L_1 \to f^*E\to L_2\to 0$$
 with $L_i$ line bundles and $\frac{\mu(L_2)-\mu(f^*E)}{\deg f}<\epsilon/p$.
 Now consider the curve $D:= f^{-1}C$ in $\tilde{X}$. Since $C$ is general, by \cite{Zhang} $D$ is geometrically unibranched.
 Hence $D$ is irreducible. Let $\pi:\tilde{D}\to D_{red}$ be the normalization of $D_{red}$. Let $\psi$ denote the composite morphism
 $\tilde{D}\to D_{red}\to D \to C$. Note that $\psi$ is also finite.
 
 On $\tilde{D}$ we have

 $$0\to \psi^*F_1\to \psi^*F_C^{*}E|_{C}$$
 and $\deg \psi^*F_1-\deg \psi^*F_C^{*}E|_{C}=\deg \psi (\deg F_1-\deg F_C^{*}E|_{C})>0$. 
 
 Hence $\psi^*F_C^{*}E|_{C}$ is not semistable and $E'=\psi^*F_1$ is  a destabilizing subbundle of  $\psi^*F_C^{*}E|_{C}$.
 
 Let $\tau$ denote the morphism $\tilde{D}\to D_{red}\to D$.
 On $\tilde{D}$, we also have an exact sequence
 $$0\to F_{\tilde{D}}^*(\tau^*(L_1|_D)) \to F_{\tilde{D}}^*(\tau^*(f^*E|_D))\to F_{\tilde{D}}^*\tau^*(L_2|_D))\to 0$$
 
 Now  %using $\mu(L_2|D)=\mu(\pi^*L_2|_D)$ and $\deg \pi\circ f=\deg f$, we see that
 \begin{eqnarray*}
 \frac{1}{\deg \psi}(\mu(E')-\mu(F_{\tilde{D}}^*(\tau^*L_2|_D)))
 %&=&\frac{1}{\deg \psi}(\mu(E')-\mu(F^*L_2|_D))\\
 &=&\frac{\mu(E')-\mu(\psi^*F_C^*E|_C)+\mu(\psi^*F_C^*E|_C)-\mu(F_{\tilde{D}}^*\tau^*L_2|_D)}{\deg \psi}\\
 &=&\frac{\mu(E')-\mu(\psi^*F_C^*E|_C)+\mu(F_{\tilde{D}}^*\tau^*f^*E|_C)-\mu(F_{\tilde{D}}^*\tau^*L_2|_D)}{\deg \psi}\\
 &=&\mu(F_1)-\mu(F_C^*E|_C)-\frac{p\deg \tau [\mu(f^*E)-\mu(L_2)]}{\deg \psi}\\
 &>&\mu(F_1)-\mu(F_C^*E|_C)-\epsilon\\
 &>&0\\
 \end{eqnarray*}
 where the third equality follows from Proposition $6$ of \cite{Kleiman}.
 Hence $\Hom(E',F_{\tilde{D}}^*\tau^*L_2|_D)=0$.
 Therefore $0\to E'\to \tau^*F_D^*L_1|_D$ and $\mu(E')< \mu(\tau^*F_D^*L_1|_D)<\mu(\tau^*F_D^*(f^*E)|_D)$, contradiction. Hence $F_C^*E|_C$ is semistable for a general smooth $C\in |H|$.
 \end{proof}
 
 \begin{remark}
 \begin{enumerate}
  
  \item
  Now suppose $E$ is a rank $2$ strongly semistable  vector bundle such that
  $\nu_1(E)=\mu(E)$, then $E|_C$ is strongly semistable for a general smooth $C\in|H|$; in particular $E|_C$ is semistable for a general smooth $C\in|H|$. But this not true in general (see Example~\ref{TP^2}). Hence we can not expect $\nu_1(E)=\mu(E)$ for arbitrary rank $2$ strongly semistable bundle.
  
 In general, if $E$ is a strongly semistable vector bundle on a polarized variety $(X,H)$, then whether $E|_C$ is strongly semistable for a very general hypersurface $C$ is an open question.   However by Theorem $3.1$ of \cite{Langer2},  it is known that if $E$ is a strongly stable vector bundle with $\Delta(E)=0$, then $E|_C$ is also strongly stable  for all smooth curve $C\in |H|$. In the  next section we study asymptotic slope for strongly semistable vector bundles with zero discriminant (i.e. $\Delta(E)=0$).
 \item One might expect to get a generalization of Theorem~\ref{strong-restriction} for a vector bundle of arbitrary rank. But at present we are unable to extend the Theorem~\ref{strong-restriction} for  vector bundles of arbitrary rank.
 \end{enumerate}
 \end{remark}

\section{Asymptotic slopes and strong semistability}
 Here in this section we prove analogue of Theorem~\ref{Main-result-curve} for
 strongly semistable vector bundles $E$ of arbitrary rank with zero discriminants i.e. $\Delta(E)=0$ and $c_1(E)=H$. In order to do this we appeal to the  Kodaira type vanishing theorem in characteristic $p>0$, for this we assume some additional condition  (which will be clear from the following) on the polarized surface and on the vector bundle.

Let $X$ be a smooth projective surface over an algebraically closed field $\mathbb{K}$ of characteristic $p>0$. Let $H$ be an ample line bundle.
Let $E$ be a strongly semistable vector bundle of rank $r\geq 2$ on $X$ with respect to the polarization $H$   with $\Delta(E)=0$.
Let $\pi$ denote the natural morphism $\Gr(k,E)\to X$.
We assume that $X, E, H$ admit a lifting $\overline{X}$, $\overline{E}$ and $\overline{H}$ respectively  to $W_2(\mathbb{K})$. Then $\Gr(k,\overline{E})$ is a lifting of $\Gr(k,E)$ to $W_2(\mathbb{K})$.
Indeed since $\overline{E}$ is a lifting, then $\overline{E}\times_{\overline{X}} X\simeq E$ and there is a  natural injection $X\hookrightarrow \overline{X}$. Hence 
$$\Gr(k,\overline{E})\times_{\overline{X}}X=\Gr(k,\overline{E}|_X)=\Gr(k,E).$$ Hence $\Gr(k,\overline{E})$ is a lifting of
$\Gr(k,E)$.

Now we state our main result of this section.

\begin{theorem}
\label{main3}
 Let  $E$ be a strongly semistable vector bundle  of rank $r\geq 2$  with $\Delta(E)=0$ on a smooth polarized surface $(X,H)$ such that $X, E, H$ admit a liftings $\overline{X}$, $\overline{E}$, $\overline{H}$ respectively to $W_2(\mathbb{K})$. Moreover if $c_1(E)=H$, then $E$ is strongly semistable if and only if  $\nu_k(E)=\mu(E)$ for all $k$. 
\end{theorem}
\begin{proof}
Here we prove the if direction and give an outline of the only if direction. 

Suppose $\nu_k(E)=\mu(E)$ for all $k$. We will show that $E$ is strongly semistable. If not, there exists a finite morphism $f:\tilde{X}\to X$ and a subsheaf $W\to f^* E$ of rank $k$ such that $\mu(W)>\mu(f^*E)$.  Hence $\nu_k(E)>\mu(E)$, which is a contradiction.

Now we give an outline of the proof of only if direction.
To prove the only if direction of the theorem  we  construct smooth surfaces $f_n:\tilde{X_n}\to X$ and subbundles $V_n\subset f_n^*(E)$  such that $\frac{\mu(V_n)}{\deg f_n}$ converges to  $\mu(E)$.  

To find such surfaces, we consider $\Gr(k,E)$ and let $\pi:\Gr(k,E)\to X$ denote the natural morphism. Let $\mathcal{N}=\{m^2: m\in \mathbb{N}  \textrm{ and } m \textrm{ is divisible by } rH^2\}$. For any $n\in \mathcal{N}$, define  
$\mathcal{L}_n:=-nk\frac{\mu(E)}{H^2}H+n^{1/2}H$.
 Note that  $\mathcal{L}_n$ is a genuine line bundle for all $n\in \mathcal{N}$. We show that $\mathcal{O}_{\Gr(k, E)}(n)\otimes \pi^*\mathcal{L}_n$ is  very ample line bundle for all large $n\in\mathcal{N}$. Then cutting down by appropriate $k(r-k)$ sections of $\mathcal{O}_{\Gr(k,E)}(n)\otimes \pi^*\mathcal{L}_n$ we get desired surfaces. 
 
 %for example, when $\epsilon=1/2$, one may choose $n=m^2$ and $m$ is divisible by $rH^2$. Let $\mathcal{N}$ denote all those natural numbers for which $\mathcal{L}_n$ is a line bundle. Note that $\mathcal{N}$ is an infinite set.
 
 The strategy to show very ampleness of the line bundles $\mathcal{O}_{\Gr(k,)}(n)\otimes \pi^*\mathcal{L}_n$ is make use of Kodaira vanishing theorem in characteristic $p$ and  Lemma $2.3$ of \cite{BP}. Let $C\in |H|$ be any smooth curve, then $E|_C$ is also strongly semistable by Proposition~\ref{generalisation-of-Langer}. Also 
 $\mu(\mathcal{L}_n|_C)=-nk\mu(E)+n^{1/2}H\cdot H$, hence
  for large $n\in \mathcal{N}$, $\mu(\mathcal{L}_n)>2g-nk\mu(E)$ where $g$ is the genus of $C$, hence by Lemma $2.3$ of \cite{BP}, $\mathcal{O}_{\Gr(k,E|_C)}(n)\otimes \pi^*\mathcal{L}_n|_C$ is very ample on $\pi^{-1}(C)=\Gr(k,E|_C)$. We use the very ampleness of
  $\mathcal{O}_{\Gr(k,E|_C)}(n)\otimes \pi^*\mathcal{L}_n|_C$ to get very ampleness of $\mathcal{O}_{\Gr(k,E)}(n)\otimes \pi^*\mathcal{L}_n$. 
  \end{proof}
  In order to complete the proof of Theorem~\ref{main3} we need several results (up to Proposition 4.10).
  The following two  lemmas might be known to experts, but here we include a proof for the convenience  to the readers.
\begin{lemma}
\label{ample}
 Let $X$ be a smooth projective surface over a field of characteristic
 $p>0$. Let $H$ be an ample line bundle and $V$ be a strongly semistable vector bundle of rank $r$ with 
 $\Delta(V)=0$. Then the line bundle $\mathcal{O}(1)$  on $\Gr(k,V)$ is ample if and only if $\det(V)$ is ample.
\end{lemma}
\begin{proof}
First we prove the lemma for $k=1$. Then $\Gr(1,V)=\mathbb{P}(V)$.
 Note that $\Sym^r V\otimes \det(V)^*$ is also strongly semistable and $\Delta(\Sym^r V\otimes \det(V)^*)=\Delta(\Sym^r V)=0$. 
 Also 
 \begin{eqnarray*}
 c_1(\Sym^r V\otimes \det(V)^*)&=&\binom{2r-1}{r}c_1(V)-\rk(\Sym^r(E))c_1(\det(V))\\
 &=&\binom{2r-1}{r}c_1(V)-\binom{2r-1}{r}c_1(V)\\
 &=&0.
 \end{eqnarray*} Hence $c_2(\Sym^r V\otimes \det(V)^*)=0$. Therefore by Proposition $5.1$ of \cite{Langer2},
 $\Sym^r V\otimes \det(V)^*$ is nef. Now since $\det(V)$ is ample
 $\Sym^r V=\Sym^r V\otimes \det(V)^*\otimes \det(V)$ is also ample.
 Hence $V$ is ample by Proposition $2.4$ of \cite{Hartshorne-ample}.

 Now we prove the lemma for $k>1$. Note that $\Gr(k,V)$ embeds in $\mathbb{P}(\Lambda^kV)$ by  Pl\"{u}cker embedding and $\mathcal{O}_{\mathbb{P}(\Lambda^kV)}(1)$ pulls back to $\mathcal{O}_{\Gr(k,V)}(1)$. Hence in order to show $\mathcal{O}_{\Gr(k,V)}(1)$ ample it is enough to show that $\mathcal{O}_{\mathbb{P}(\Lambda^kV)}(1)$ is ample. Now as $V$ is strongly semistable, then $\Lambda^kV$ is so.
 Also $\Delta(\Lambda^k V)=0$, since $\Delta(V)=0$ (by Lemma~\ref{discriminant}). Then by above paragraph $\mathcal{O}_{\mathbb{P}(\Lambda^kV)}(1)$ is ample.  Hence the lemma.
 
 %\textbf{ALTERNATE PROOF:}\\
 %First we prove the lemma for $k=1$. Then $\Gr(1,V)=\mathbb{P}(V)$.
 %First note that $\Sym^r V\otimes \det(V)^*$ is also strongly semistable and $\Delta(\Sym^r V\otimes \det(V)^*)=\Delta(\Sym^r V)=0$. 
 %Also 
 %\begin{eqnarray*}
 %c_1(\Sym^r V\otimes \det(V)^*)&=&\binom{2r-1}{r}c_1(V)-\rk(\Sym^r(E))c_1(\det(V))\\
 %&=&\binom{2r-1}{r}c_1(V)-\binom{2r-1}{r}c_1(V)\\
 %&=&0.
 %\end{eqnarray*}
 %Hence $c_2(\Sym^r V\otimes \det(V)^*)=0$. Therefore by Proposition $5.1$ of \cite{Langer2},
 %$\Sym^r V\otimes \det(V)^*$ is nef. Now since $\det(V)$ is ample
 %$\Sym^r V=\Sym^r V\otimes \det(V)^*\otimes \det(V)$ is also ample.
 %Hence $V$ is ample by Proposition $2.4$ of \cite{Hartshorne-ample}.
 
 %Now we prove the lemma for $k>1$. Note that $\Gr(k,V)$ embeds in $\mathbb{P}(\Lambda^kV)$ by  Pl\"{u}cker embedding and $\mathcal{O}_{\mathbb{P}(\Lambda^kV)}(1)$ pulls back to $\mathcal{O}_{\Gr(k,V)}(1)$. Hence in order to show $\mathcal{O}_{\Gr(k,V)}(1)$ ample it is enough to show that $\mathcal{O}_{\mathbb{P}(\Lambda^kV)}(1)$ is ample. Since $V$ is strongly semistable, then
 %$V^{\otimes k}$ is so. Also $\Delta(V^{\otimes k})=sr^{2(s-1)}\Delta(V)=0$, by $(3.4)$ of \cite{Lehn}. Hence what we prove in above $V^{\otimes k}$ is ample. Now $\Lambda^k(V)$ being quotient of $V^{\otimes k}$ is also ample by Proposition $2.2$ of \cite{Hartshorne-ample}. That is $\mathcal{O}_{\mathbb{P}(\Lambda^kV)}(1)$ is ample and hence the lemma.
 
\end{proof}
\begin{lemma}
\label{det-of-tensor}
 Let $E$ be vector bundle on $X$, then $$c_1(E^{\otimes n})=nr^{n-1}c_1(E).$$
\end{lemma}

 \begin{proof}
The proof follows from the repeated application of the following formula:
given two vector bundles $V_1$ and $V_2$ of rank $r_1$ and $r_2$ respectively, $$c_1(V_1\otimes V_2)= r_2 c_1(V_1)+ r_1 c_1(V_2).$$
 \end{proof}
The following Theorem plays a crucial role in proving very ampleness of $\mathcal{O}_{\Gr(k,E}(n)\otimes \pi^*\mathcal{L}_n$.
\begin{theorem}
\label{H^1-vanishing}
With the same hypothesis as in Theorem~\ref{main3},
fix $m\geq 1$.
For any smooth curve $C\in |mH|$,
 the cohomology module $$H^1(\Gr(k,E), \mathcal{O}(n)\otimes \pi^* \mathcal{L}_n(-C))=0$$ 
 for large $n\in \mathcal{N}$, where $\mathcal{L}_n=-\frac{nk}{r}H+n^{1/2} H$.
\end{theorem}

\begin{proof}
In the proof of this theorem we will make use to Kodaira type vanishing theorem in characteristic $p>0$.
We prove the theorem in two steps.\\
Let $L$ be an ample line bundle on $X$. Then in the first step we will show that $\mathcal{O}_{\Gr(k,E)}(n)\otimes\frac{-nk}{r} \pi^*H\otimes\pi^*L$ is ample for large $n\in \mathcal{N}$. In the next step using first step  and  Kodaira type vanishing theorem  along with Serre duality we will conclude the theorem.

$\mathbf{Step 1:}$ 
We have the following commutative diagram
$$\xymatrix{
\Gr(k,E) \ar^{\pi}[rd] \ar@{^{(}->}^i[r]
& \mathbb{P}(\Lambda^k(E))\ar^{\pi_1}[d]\\
&X}$$
 such that $\mathcal{O}_{\Gr(k,E)}(n)\otimes\frac{-nk}{r}\pi^*H\otimes \pi^* L=i^*(\mathcal{O}_{\mathbb{P}(\Lambda^k(E))}(n)\otimes\frac{-nk}{r} \pi_1^*H\otimes \pi_1^*L)$. Hence in order to show 
 $\mathcal{O}_{\Gr(k,E)}(n)\otimes\frac{-nk}{r}\pi^*H\otimes\pi^*L$ is ample on $\Gr(k,E)$
 it is sufficient to show that $\mathcal{O}_{\mathbb{P}(\Lambda^k(E))}(n)\otimes\frac{-nk}{r} \pi_1^*H\otimes \pi_1^*L$ is ample on $\mathbb{P}(\Lambda^k{E})$.
 
Using Lemma~\ref{det-of-tensor}, one can check that $\det(E^{\otimes nk}\otimes\frac{-nk}{r}H\otimes L)$ is ample. Since $\Delta (E)=0$, hence $\Delta(E^{\otimes nk})=0$, and thus by Proposition~\ref{discreminant-properties}$(1)$ we have

$$\Delta(E^{\otimes nk}\otimes\frac{-nk}{r}H\otimes L)=0.$$

Hence by Lemma~\ref{ample}, $E^{\otimes nk}\otimes\frac{-nk}{r}H\otimes L$ is ample.
Since quotient of an ample bundle is ample, it follows that $\Sym^n(\Lambda^k(E))\otimes\frac{-nk}{r}H\otimes L$ is also ample.

Now consider the following commutative diagram

$$\xymatrix{
\mathbb{P}(\Lambda^k(E)) \ar^{\pi_1}[rd] \ar@{^{(}->}^i[r]
& \mathbb{P}(\Sym^n(\Lambda^k(E)))\ar^{\tilde{\pi_1}}[d]
&\mathbb{P}(\Sym^n(\Lambda^k(E))\otimes\frac{-nk}{r}H\otimes L)
\ar[l]\ar^{\tilde{\tilde{\pi_1}}}[ld]\\
&X}$$
such that 
\begin{eqnarray*}
\mathcal{O}_{\mathbb{P}(\Lambda^k(E))}(n)\otimes\frac{-nk}{r}\pi_1^*H\otimes \pi_1^*{L}&
=&i^{*}[\mathcal{O}_{\mathbb{P}(\Sym^n(\Lambda^k(E)))}(1)\otimes\frac{-nk}{r}{\tilde{\pi_1}}^*H\otimes {\tilde{\pi_1}}^*L]\\
&=&i^{*}[\mathcal{O}_{\mathbb{P}(\Sym^n(\Lambda^k(E)))}(1)]\otimes\frac{-nk}{r}\pi_1^*H\otimes \pi_1^*L.
\end{eqnarray*}
Since $\mathcal{O}_{\mathbb{P}(\Sym^n(\Lambda^k(E))\otimes\frac{-nk}{r}H\otimes L)}(1)$ is ample on $\mathbb{P}(\Sym^n(\Lambda^k(E))\otimes\frac{-nk}{r}H\otimes L)$,
$\mathcal{O}_{\mathbb{P}(\Sym^n(\Lambda^k(E))}(1)\otimes\frac{-nk}{r}\pi_1^*H\otimes\tilde{\pi_1}^*L$ is ample on $\mathbb{P}(\Sym^n(\Lambda^k(E))$.
Hence   $\mathcal{O}_{\mathbb{P}(\Lambda^k(E))}(n)\otimes\frac{-nk}{r}\pi_1^*H\otimes \pi_1^*{L}$ is ample on $\mathbb{P}(\Lambda^k(E))$.
Thus  $\mathcal{O}_{\Gr(k,E)}(n)\otimes\frac{-nk}{r}\pi^*H\otimes\pi^*{L}$ is ample on 
$\Gr(k,E)$.

$\mathbf{Step 2:}$
Let $C\in |mH|$ be any smooth curve. Consider the short exact sequence

$$0\to\mathcal{O}(-\pi^{-1}C)\to \mathcal{O}_{\Gr(k,E)}\to 
\mathcal{O}_{\pi^{-1}C}\to 0.$$
which gives long exact sequence  in homology modules,

$$
 %\label{les}
 0\to H^0(\Gr(k,E), \mathcal{O}(n)\otimes \pi^*(\mathcal{L}_n(-C)))\to H^0(\Gr(k,E), \mathcal{O}(n)\otimes \pi^*\mathcal{L}_n)\to\hphantom{njnkbhkb} $$
$$H^0(\Gr(k,E|_C), \mathcal{O}(n)\otimes \mathcal{L}_n|_{\pi^{-1}C})\to H^1(\Gr(k,E), \mathcal{O}(n)\otimes \pi^* \mathcal{L}_n(-C)) \to \cdots$$

Now we claim that $ H^1(\Gr(k,E), \mathcal{O}_n\otimes \pi^*(\mathcal{L}_n(-C)))=0$ for all large enough $n\in \mathcal{N}$.

Note that $\mathcal{O}_{\Gr(k,E)}(n)\otimes\pi^*(\mathcal{L}_n(-C))=\mathcal{O}_{\Gr(k,E)}(n)\otimes\pi^*\frac{-nk}{r}H\otimes \pi^*[n^{1/2}H\otimes\mathcal{O}(-C)]$.
Now we can choose $n \in \mathcal{N}$ and $n\gg 0$  such that, $n^{1/2}H\otimes \mathcal{O}(-C)$ is ample. Take $L_n=n^{1/2}H\otimes \mathcal{O}(-C)$. Hence by Step 1, 
$\mathcal{O}_{\Gr(k,E)}(n)\otimes \pi^*({\mathcal{L}_n(-C)})$ is ample for all $n\gg 0$ with $n \in \mathcal{N}$.
Again we can choose $n\in N$ sufficiently large such that 
$\mathcal{O}_{\Gr(k,E)}(n)\otimes \pi^*({\mathcal{L}_n(-C)})\otimes K^*_{\Gr(k,E)}$  is also ample. This can be seen as follows:  since
$\Delta(E)=0$ and $\det(E)$ is ample,
by Lemma~\ref{ample},  
$\mathcal{O}_{\Gr(k,E)}(1)$ is ample on $\Gr(k,E)$.
Hence for $K_{\Gr(k,E)}^*$, there exists $s$ divisible by $r(H.H)$ such that $\mathcal{O}_{\Gr(k,E)}(s)\otimes K_{\Gr(k,E)}^*$ is ample. We can choose $n\in \mathcal{N}$ large such that $(n^{1/2}-\frac{sk}{r})H(-C)$ is ample. Hence
$$
 \mathcal{O}_{\Gr(k,E)}(n)\otimes \pi^*({\mathcal{L}_n(-C)})\otimes K^*_{\Gr(k,E)}\phantom{asdfgjklcdefrvgbthnyjxcddvfffffbgeeeeee}$$
$$=\mathcal{O}_{\Gr(k,E)}(n-s)\otimes \pi^*\frac{-(n-s)k}{r}H\otimes\pi^*((n^{1/2}-\frac{sk}{r})H(-C))\otimes\mathcal{O}_{\Gr(k,E)}(s)\otimes K^*_{\Gr(k,E)}$$

 Now both the  line bundles $\mathcal{O}_{\Gr(k,E)}(n-s)\otimes \pi^*\frac{-(n-s)k}{r}H\otimes\pi^*((n^{1/2}-\frac{sk}{r})H(-C))$, $\mathcal{O}_{\Gr(k,E)}(s)\otimes K^*_{\Gr(k,E)}$ are ample, hence $\mathcal{O}_{\Gr(k,E)}(n)\otimes \pi^*({\mathcal{L}_n(-C)})\otimes K^*_{\Gr(k,E)}$  is also ample for large $n$.
Then Kodaira type vanishing theorem says \cite{Esnault}, $ H^i(\Gr(k,E), \mathcal{O}(-n)\otimes (\pi^*(\mathcal{L}_n(-C)))^*\otimes K_{\Gr(k,E)})=0$ for all large enough $n\in \mathcal{N}$ and for all $i<\dim \Gr(k,E)$. Hence by Serre duality
the claim follows.
\end{proof}

 Next we prove that $\mathcal{O}_{\Gr(k,E)}(n)\otimes \pi^* \mathcal{L}_n$ is very ample,  for all large $n\in \mathcal{N}$.  But first  we show 
 a useful lemma.
 
 \begin{lemma}
 \label{lemma-3}
 Let $X$ be any projective variety and $H$ be a very ample divisor on $X$.  Let $x\in X$,  and $\tau
 :Bl_x(X)\to X$ denote the blow up of $X$ at $x$ with $E$ as exceptional divisor. Then $2\tau^*H-E$ is very ample on $Bl_x(X)$. %(COMMENT: I THINK WE NEED  $X$ SMOOTH OTHERWISE THE EXCEPTIONAL $E$ NEED NOT BE A DIVISOR)
\end{lemma}

\begin{proof}
 Since $H$ is very ample, with respect to a fixed embedding $X$ can be realized as a closed subvariety of $\mathbb{P}^N$   for some $N\in \mathbb{N}$.
 Hence $Bl_x(X)\subseteq Bl_x(\mathbb{P}^{N})$. Let $\tilde{E}$ denote the exceptional divisor of 
 $\tilde{\tau}: Bl_x(\mathbb{P}^N)\to \mathbb{P}^N$. Now in order to show that $2\tau^*H-E$ is very ample on $Bl_x(X)$, it is enough to show that $2\mathcal{O}_{\mathbb{P}^N}(1)-\tilde{E}$ is very ample on $Bl_x(\mathbb{P}^N)$.
 
 Note that $\mathcal{O}_{\mathbb{P}^N}(1)-\tilde{E}$ gives the projection morphism from $Bl_x(\mathbb{P}^N)\to \mathbb{P}^{N-1}$ and 
 $\mathcal{O}_{\mathbb{P}^N}(1)$ gives the morphism from 
 $Bl_x(\mathbb{P}^N)\to \mathbb{P}^{N}$. Hence together 
 $2\tilde{\tau}^*\mathcal{O}_{\mathbb{P}^N}(1)-\tilde{E}$ gives a morphism 
 $Bl_x(\mathbb{P}^N)\to \mathbb{P}^N\times\mathbb{P}^{N-1}$ which is the natural morphism of $Bl_x(\mathbb{P}^N)\hookrightarrow \mathbb{P}^N\times\mathbb{P}^{N-1}$.
 Hence $2\tilde{\tau}^*\mathcal{O}_{\mathbb{P}^N}(1)-\tilde{E}$ is very ample on $Bl_x(\mathbb{P}^N)$.
\end{proof}

\begin{theorem}
 The line bundle $\mathcal{O}_{\Gr(k,E)}(n)\otimes \pi^* \mathcal{L}_n$ is very ample,  for all large $n\in \mathcal{N}$, where $\mathcal{L}_n=-\frac{nk}{r}H+n^{1/2} H$.
\end{theorem}

\begin{proof}
In order to show that the line bundle $\mathcal{O}_{\Gr(k,E)}(n)\otimes \pi^* \mathcal{L}_n$ is very ample, we need to show that the line bundle separates points and separates tangent vectors.

First we will show that  $\mathcal{O}_{\Gr(k,E)}(n)\otimes \pi^* \mathcal{L}_n$ separates points.
Take two points $z_1,z_2\in {\Gr(k,E)}$. Let $\pi(z_1)=x,\pi(z_2)=y$. By Theorem $3.1$ of \cite{SB}, one can choose a smooth curve $C\in|mH|$ that contains $x,y$.
Consider the short exact sequence of sheaves:
$$0\to \mathcal{O}_{\Gr(k,E)}(n)\otimes \pi^* \mathcal{L}_n(-C)\to \mathcal{O}_{\Gr(k,E)}(n)\otimes \pi^* \mathcal{L}_n\to 
\mathcal{O}_{\Gr(k,E|_C)}(n)\otimes \pi^* \mathcal{L}_n|_{\Gr(k,E|_C)}\to  0$$
which yields the following long exact sequence in homology:
$$\cdots\to H^0(\Gr(k,E), \mathcal{O}(n)\otimes \pi^* \mathcal{L}_n)
\to H^0(\Gr(k,E|_C), \mathcal{O}(n)\otimes \pi^* \mathcal{L}_n)$$
$$\to H^1(\Gr(k,E), \mathcal{O}(n)\otimes \pi^* \mathcal{L}_n(-C)) \to \cdots\phantom{asdfgjklcdefrvgbrre}$$
By Theorem~\ref{H^1-vanishing},
  $$H^1(\Gr(k,E), \mathcal{O}(n)\otimes \pi^* \mathcal{L}_n(-C))=0.$$
Hence the morphism
$$\cdots\to H^0(\Gr(k,E), \mathcal{O}(n)\otimes \pi^* \mathcal{L}_n)
\to H^0(\Gr(k,E|_C), \mathcal{O}(n)\otimes \pi^* \mathcal{L}_n)$$
 is surjective.
 By Lemma $2.3$ of \cite{BP}, $\mathcal{O}_{\Gr(k,E|_C)}(n)\otimes \pi^* \mathcal{L}_n$ separates points on $\Gr(k,E|_C)$. Hence there exists a section $\sigma$  in $\mathcal{O}_{\Gr(k,E|_C)}(n)\otimes \pi^* \mathcal{L}_n$  such that $\sigma(z_1)=0$ and $\sigma(z_2)\neq 0$. Then choose a lift $\tilde{\sigma}\in H^0(\Gr(k,E), \mathcal{O}(n)\otimes \pi^* \mathcal{L}_n)$  of $\sigma$ which also has the property that $\tilde{\sigma}(z_1)=0$ and $\tilde{\sigma}(z_2)\neq 0$.
 Hence $\mathcal{O}_{\Gr(k,E)}(n)\otimes \pi^* \mathcal{L}_n$ separates points.
 
Now we will show that $\mathcal{O}_{\Gr(k,E)}(n)\otimes \pi^* \mathcal{L}_n$
separates tangent vectors of $\Gr(k,E)$.

Let $\tilde{\pi}: T_z(\Gr(k,E))\to T_x(X)$ be the natural morphism  induced by $\pi:\Gr(k,E)\to X$ with kernel $T_z(\Gr(k,E|_x))$.
Let $v\in T_z(\Gr(k,E))$.

$\mathbf{Case ~1:}$ Suppose that $\tilde{\pi}(v)=0\in T_x(X)$.
Then $v\in T_z(\Gr(k,E|_x))$. Choose a smooth curve $C\in |H|$ such that $x\in C$. Then 
$v\in T_z(\Gr(k,E|_x))\subseteq T_z(\Gr(k,E|_C))$.
Since by Lemma $2.3$ of \cite{BP}, on $\Gr(k,E|_C)$, 
$\mathcal{O}_{\Gr(k,E|_C)}(n)\otimes \pi^* \mathcal{L}_n$ separates tangent vectors, there exists a section 
$\sigma\in H^0(\Gr(k,E|_C, \mathcal{O}(n)\otimes \pi^* \mathcal{L}_n)$ such that $\sigma(z)=0$ and $v\notin T_z(\divi (\sigma))$. Since the natural morphism
$$H^0(\Gr(k,E), \mathcal{O}(n)\otimes \pi^* \mathcal{L}_n)
\to H^0(\Gr(k,E|_C), \mathcal{O}(n)\otimes \pi^* \mathcal{L}_n)$$
is surjective, there exists a lift 
$\tilde{\sigma}
\in H^0(\Gr(k,E), \mathcal{O}(n)\otimes \pi^* \mathcal{L}_n)$ of $\sigma$. Note that $\tilde{\sigma}(z)=0$ and also $v\notin T_z(\divi(\tilde{\sigma}))$. Hence in this case $\mathcal{O}_{\Gr(k,E)}(n)\otimes \pi^* \mathcal{L}_n$ separates tangent vectors.

$\mathbf{Case~2:}$ Suppose that $\tilde{\pi}(v)=w\neq 0$. We will show that there exists a smooth $D\in |mH|$, with $\frac{m}{2}H$ ample such that $w\in T_x(D)$. Hence $v\in T_z(\Gr(k,E|_D))$, and then the rest of the arguments follow from case 1.

To show such smooth $D\in |mH|$ with  $w\in T_x(D)$ exists  
let us consider the blow up morphism $\tau:Bl_x(X)\to X$ where $Bl_x(X)$ denotes the blow up of $X$ at $x$ with $E$ as the exceptional divisor. One can note that $E=\mathbb{P}(T_x(X))$. Let $y$ represent $w$ in $E$. Also one can assume that $\frac{m}{2}H$ is very ample and hence by  Lemma~\ref{lemma-3}, $\tau^*mH-E$ is very ample on $Bl_x(X)$.  Let $C$ be a smooth curve in $\tau^*mH-E$ passing through $y$ (such smooth curve exists by the proof of Theorem $3.1$ of \cite{SB}). Since $(\tau^*mH-E).E=1$, $\tau(C)\in |mH|$ is a smooth curve in $X$ and $w\in T_x(\tau(C))$.  Hence the proof.

\end{proof}

 Now since $\mathcal{O}_{\Gr(k,E)}(n)\otimes \pi^* \mathcal{L}_n$ is very ample for all large $n\in \mathcal{N}$,
 we can choose $k(r-k)$ sections such that they cut down $\Gr(k,E)$ into a smooth surface $\tilde{X}_n$. Next we show that $\tilde{X}_n$ is finite over $X$ for all large $n \in \mathcal{N}$. First we prove a general proposition concerning a general hyperplane section of a flat family is flat.
 \begin{proposition}
 \label{flatness-general-hyperplane-section}
  Let $f: X\to Y$ be a flat family obtained by a morphism $f$ from a smooth variety $X$ to a smooth surface $Y$. Let $X\subset \mathbb{P}^N$  be an embedding obtained by an very ample divisor $\mathcal{L}$  such that the general fiber is not a linear subspace of $\mathbb{P}^N$.  Then for a general hyperplane $H$, $X\cap H\to Y$ is also a flat family induced by $f$.
 \end{proposition}
 
 \begin{proof}
 For each $y\in Y$, let $X_y$ denote the fibre.
  Consider the incidence variety $S=\{(y,H):  X_y\subset H\}\subset Y\times \check{\mathbb{P}}^N$. Let $p_1$ and $p_2$ denote the projections from $S$ to $Y$ 
  and  to $\check{\mathbb{P}}^N$ respectively. Then the fibre of $p_1$ over a point $y\in Y$ denoted by
  $S_y=\{H\in \check{\mathbb{P}}^N: X_y\subset H\}$. Note that $S_y$  is a linear subspace of $\check{\mathbb{P}}^N$ of dimension $N-\dim <X_y>-1$, where $<X_y>$ denotes the smallest linear subspace containing $X_y$ (i.e. the linear span of $X_y$). Since for each $y\in Y$, $ \dim S_y\leq  N - \dim X_y -1 $ then $\dim S\leq  N - \dim X_y -1 + \dim Y$.  The image  $p_2(S)$ is a proper closed subset of $\check{\mathbb{P}}^N$ unless $\dim X_y\leq 1$.  Hence for the case when each $y\in Y$, $\dim X_y\geq 2$, there will be a general hyperplane $H$ such that it contains no fibre $X_y$, in other words, for each $y$, $\dim X_y\cap H=\dim X_y-1$. By Bertini's theorem, shrinking the open set if required, we may assume $X\cap H$ is also smooth, hence in this case $X\cap H\to Y$ is a morphism induced from $f$ between smooth varieties with equidimensional fibres, hence $X\cap H\to Y$ is a flat family. Now we consider the case when $\dim X_y=1$, for some $y$ i.e. $X_y$ is a curve.  Consider $S_1=\{(y,H)\in S: X_y \textrm{ is  linearly embedded in } \mathbb{P}^1\}$. Since  the general fiber is not a linear subspace of $\mathbb{P}^N$, $p_1(S_1)$ is either a finite set of points or a curve $C$ in $Y$. Hence $\dim S_1=N-\dim X_y-1+\dim C=N-1$,  where $y\in C$, again $p_2(S_1)$ is a proper closed subset of $\check{\mathbb{P}}^N$, thus arguing as in the previous case one checks that   for a general hyperplane $H$, $X\cap H\to Y$ is a flat family.
 \end{proof}
 
 Next we give an example which shows that the hypothesis  that the general fiber is not a linear subspace of $\mathbb{P}^N$ is necessary.
 \begin{ex}
  Consider $X=SL_3/B$ which embeds in  $\mathbb{P}^8$, $Y=\mathbb{P}^2$. Let $f: X\to Y$ denote the morphism which sends each full flag to its linear subspace. One notes that each fibre is a linear space $\mathbb{P}^1$ in $\mathbb{P}^8$. It is known that 
  for a general hyperplane $H$, $X\cap H\to Y$ is birational. If $X\cap H\to Y$ is a flat family. it would be a finite map and hence an isomorphism, which is not true as we show next that there is no section from $\mathbb{P}^2\to SL_3/B$. Let $D_1$ and $D_2$ be the divisors in $SL_3/B$. They are pull backs of lines from $\mathbb{P}^2$ and $\mathbb{P}^2$ dual. Since the square of a line in $\mathbb{P}^2$ is a point, we get $D_1^2=L_1$ and $D_2^2 = L_2$, where $L_1$ and $L_2$ are fibres. The very ample divisor $H$ on $SL_3/B$ is $D_1 + D_2$. If we intersect these divisor $D_1 \cdot D_2 = L_1 + L_2$. Hence by squaring we get:

 $$(D_1+D_2)^2 = D_1^2 + D_2^2 + 2 D_1\cdot D_2 = 3 D_1\cdot D_2.$$

Assume there is a map $f$  from $\mathbb{P}^2$ to $SL_3/B$ and $f^*{(D_1)}=a$ and $f^*{(D_2)} = b$.  Then $f^*{(H)} = f^*{(D_1 + D_2)} = a + b$. Hence $f^*{(H^2)}= f^*{(D_1+D_2)^2} = (a + b)^2 = a^2 + b^2 + 2ab$. Hence we get: $a^2 + b^2 + 2ab = 3ab$. subtracting $4ab$ from both sides we obtain $(a-b)^2 = -ab$, which is true only when $a=0$ and $b=0$.

 \end{ex}

%%%%%%%%%%%%%%%%%%%%%%%%%%%%%%%%%%%%%%%%%%%%%%%%%%%%%%%%%%%%%%%%%%%

% \begin{proof}
 % For each $y\in Y$, let $X_y$ denote the fibre. Define  $S_y=\{H\in \check{\mathbb{P}}^N: X_y\subset H\}$. Note that $S_y$  is a linear subspace of $\check{\mathbb{P}}^N$ of dimension $N-\dim <X_y>-1$, where $<X_y>$ denotes the smallest linear subspace containing $X_y$ (i.e. the linear span of $X_y$). By hypothesis $ \dim S_y\leq  N - \dim X_y -1 $. Hence  the total dimension of all hyperplanes that may contain some fibre is  $ N - \dim X_y -1 + \dim Y$. Hence this will never  be all hyperplanes if  $\dim Y=2 \leq  \dim X_y$.So when $\dim X_y\geq 2$, for each $y\in Y$, there is a general hyperplane $H$ such that the induced map  $X\cap H\to Y$ is  a flat family. Now we consider the case when $\dim X_y=1$, for some $y$ i.e. $X_y$ is a curve. Note that by Bertini's theorem  general fibre is irreducible and it is not linear, hence the set $\{y\in Y: \dim X_y=\mathbb{P}^1 \}$ is either a finite set or supported in a curve. One can check that for both of the cases the total dimension of hyperplanes containing some fibres which are $\mathbb{P}^1$ is strictly less than $N$. Hence in this case also we get a general hyperplane $H$ such that the induced map  $X\cap H\to Y$ is  a flat family. 
 %\end{proof}

%%%%%%%%%%%%%%%%%%%%%%%%%%%%%%%%%%%%%%%%%%%%%%%%%%%%%%%%%%%%%%%%%%%%%%%%
 \begin{theorem}
  \label{finiteness-of-X_n}
  For  all large $n\in \mathcal{N}$, the morphism $\tilde{X}_n\to X$ induced by $\pi$ is a finite morphism.
 \end{theorem}
\begin{proof}
 By Proposition \ref{flatness-general-hyperplane-section}, a general section of $\mathcal{O}_{\Gr(k,E)}(n)\otimes \pi^* \mathcal{L}_n$ cuts each fibre of $\pi: \Gr(K,E)\to X$ into a variety of dimension exactly one less. Using Proposition ~\ref{flatness-general-hyperplane-section} repeatedly one see that the morphism $\tilde{X}_n\to X$ is quasi finite. Since the morphism is also proper, it is a finite map by Zariski's main theorem.
\end{proof}

Since $\tilde{X}_n$ is complete intersection,
then  $\deg \mathcal{O}_{\Gr(k,E)}(1)|_{\tilde{X}_n}$ with respect to $\tilde{X}_n.\pi^*H$, can be calculated as the cup product of the cycle classes of the corresponding divisors  with the class of $\mathcal{O}_{\Gr(k,E)}(1)$.\\
Let $D=\tilde{X}_n\cdot \pi^*H$.

On $\Gr(k,E)$, we have the universal exact sequence:
$$0\to \mathcal{S}(E)\to \pi^*E\to \mathcal{Q}(E)\to 0.$$

Hence 
$$0\to \mathcal{S}(E)|_{\tilde{X}_n}\to \pi^*E|_{\tilde{X}_n}\to \mathcal{Q}(E)|_{\tilde{X}_n}\to 0.$$

\begin{proposition}
\label{key-proposition3}
Let $f_n:\tilde{X_n}\to X$ denote the morphism induced by $\pi$. Then 
 $$\lim\limits_{n\to\infty}\frac{\mu(f^*_nE)-\mu(\mathcal{S}(E)|_{\tilde{X}_n})}{\deg f_n}= 0,$$ where  $\mu$ is taken with respect to $D$.
\end{proposition}

\begin{proof}
\begin{eqnarray*}
 \deg \mathcal{Q}(E)|_{\tilde{X}}&=&[\mathcal{O}(n)\otimes\pi^*\mathcal{L}_n]^{k(r-k)}\cdot [\pi^*H]\cdot[\mathcal{O}(1)]\\
 &=&([\mathcal{O}(n)]+[\pi^*\mathcal{L}_n])^{k(r-k)}\cdot [\pi^*H] \cdot [\mathcal{O}(1)]\\
 &=&([\mathcal{O}(n)]^{k(r-k)}+k(r-k)[\mathcal{O}(n)]^{k(r-k)-1}\cdot[\pi^*\mathcal{L}_n])\cdot[\pi^*H]\cdot[\mathcal{O}(1)]\\
 &=&n^{k(r-k)}[\mathcal{O}(1)]^{k(r-k+1)}\cdot [\pi^*H]+k(r-k)n^{k(r-k)-1}[\mathcal{O}(1)]^{k(r-k)}\cdot[\pi^*\mathcal{L}_n]\cdot [\pi^*H]\\
 &=&n^{k(r-k)}[\mathcal{O}(1)]^{k(r-k)+1}\cdot[\pi^*H]+k(r-k)n^{k(r-k)-1}([\mathcal{O}(1)]^{k(r-k)}\cdot \mathcal{F})\deg\mathcal{L}_n\\
\end{eqnarray*}
the last equality follows from the fact that
$$[\mathcal{O}(1)]^{k(r-k)}\cdot [\pi^*\mathcal{L}_n]\cdot [\pi^*H]=
([\mathcal{O}(1)]^{k(r-k)}\cdot \mathcal{F})\deg\mathcal{L}_n,$$ where $\mathcal{F}$ denotes any fiber of $\pi:\Gr(k,E)\to X$.

\begin{eqnarray*}
 \deg \pi^*E|_{\tilde{X}}&=&[\mathcal{O}(n)\otimes\pi^*\mathcal{L}_n]^{k(r-k)}\cdot [\pi^*H]\cdot [\pi^*\det E]\\
 &=&n^{k(r-k)}[\mathcal{O}(1)]^{k(r-k)}\cdot [\pi^*H]\cdot [\pi^*\det E]+[\pi^*\mathcal{L}_n]\cdot [\pi^*H]\cdot [\pi^*\det E]\\
 &=&n^{k(r-k)}([\mathcal{O}(1)]\cdot \mathcal{F})\deg E+0\
\end{eqnarray*}
the last equality follows from the fact that
$$[\pi^*\mathcal{L}_n]\cdot[\pi^*H]\cdot[\pi^*\det E]=0.$$

Need to find $[\mathcal{O}_{\Gr(k,E)}(1)]^{k(r-k)+1}\cdot[\pi^*H]$. 
Note that if $C\in |H|$ be any smooth curve, then 
$$[\mathcal{O}_{\Gr(k,E)}(1)]^{k(r-k)+1}\cdot [\pi^*H]=[\mathcal{O}_{\Gr(k,E|_C)}(1)]^{k(r-k)+1}.$$ Hence by  Lemma $2.3$ of \cite{PS}

$$
 [\mathcal{O}_{\Gr(k,E)}(1)]^{k(r-k)+1}\cdot[\pi^*H]=(k(r-k)+1)k\mu(E)([\mathcal{O}_{\Gr(k,E)}(1)]^{k(r-k)}\cdot \mathcal{F})
 $$
 
Hence \begin{eqnarray*}
       \deg \mathcal{S}(E)|_{\tilde{X}_n}&=&\deg \pi^*E|_{\tilde{X}}-\deg\mathcal{O}(1)|_{\tilde{X}}\\
       &=&n^{k(r-k)}([\mathcal{O}(1)]^{k(r-k)}\cdot \mathcal{F})\deg E-n^{k(r-k)}[\mathcal{O}(1)]^{k(n-k)+1}\cdot [\pi^*H]\\
       &&-(k(r-k)n^{k(r-k)-1}([\mathcal{O}(1)]^{k(r-k)}\cdot \mathcal{F})\deg\mathcal{L}_n\\
       &=&n^{k(r-k)}([\mathcal{O}(1)]^{k(r-k)}\cdot \mathcal{F})\deg E-n^{k(r-k)}
       (k(r-k)+1)k\mu(E)([\mathcal{O}_{\Gr(k,E)}(1)]^{k(r-k)}\cdot \mathcal{F})\\
       &&-k(r-k)n^{k(r-k)-1}([\mathcal{O}(1)]^{k(r-k)}\cdot \mathcal{F})\deg\mathcal{L}_n\\
       &=&n^{k(r-k)}([\mathcal{O}(1)]^{k(r-k)}\cdot \mathcal{F})r\mu (E)-n^{k(r-k)}
       (k(r-k)+1)k\mu(E)([\mathcal{O}_{\Gr(k,E)}(1)]^{k(r-k)}\cdot \mathcal{F})\\
       &&-k(r-k)n^{k(r-k)-1}([\mathcal{O}(1)]^{k(r-k)}\cdot \mathcal{F})\deg\mathcal{L}_n\\
       &=&(r-k)n^{k(r-k)}([\mathcal{O}(1)]^{k(r-k)}\cdot \mathcal{F})\mu(E)-
       n^{k(r-k)}
       k(r-k)k\mu(E)([\mathcal{O}_{\Gr(k,E)}(1)]^{k(r-k)}\cdot \mathcal{F})\\
       &&-k(r-k)n^{k(r-k)-1}([\mathcal{O}(1)]^{k(r-k)}\cdot \mathcal{F})(-nk\mu(E)+n^{1/2}H\cdot H)\\
       &=&(r-k)n^{k(r-k)}([\mathcal{O}(1)]^{k(r-k)}\cdot \mathcal{F})\mu(E)-k(r-k)n^{k(r-k)-1}n^{1/2}([\mathcal{O}(1)]^{k(r-k)}\cdot \mathcal{F})H\cdot H.\\
      \end{eqnarray*}
    
    Hence 
    \begin{eqnarray*}
    \mu(\mathcal{S}(E)_{\tilde{X}_n})&=&\frac{\deg \mathcal{S}(E)|_{\tilde{X}_n}}{r-k}\\
    &=&n^{k(r-k)}([\mathcal{O}(1)]^{k(r-k)}\cdot \mathcal{F})\mu(E)-k n^{k(r-k)-1}n^{1/2}([\mathcal{O}(1)]^{k(r-k)}\cdot \mathcal{F})H\cdot H.\\
     \end{eqnarray*}
    
 Note that the degree of $f_n$ is equal to the cardinality of a general fiber of $f_n$ which equals to 
 $[\mathcal{O}_{\Gr(k,E_x)}(n)]^{k(r-k)}=n^{k(r-k)}([\mathcal{O}_{\Gr(k,E)}(1)]^{k(r-k)}\cdot \mathcal{F})$.\\
 
 Hence
 \begin{eqnarray*}
    \frac{\mu(\mathcal{S}(E)_{\tilde{X}_n})}{\deg f_n}
    &=&\frac{n^{k(r-k)}([\mathcal{O}(1)]^{k(r-k)}\cdot \mathcal{F})\mu(E)-k n^{k(r-k)-1}n^{1/2}([\mathcal{O}(1)]^{k(r-k)}\cdot \mathcal{F})H\cdot H}{n^{k(r-k)}([\mathcal{O}_{\Gr(k,E)}(1)]^{k(r-k)}\cdot \mathcal{F})}.\\
    &=& \mu(E)-\frac{kn^{1/2}H\cdot H}{n}\\
     \end{eqnarray*}
 
Therefore
\begin{eqnarray*}
\lim\limits_{n\to\infty}\frac{f_n^*E-\mu (\mathcal{S}(E)|_{\tilde{X}_n})}{\deg f_n}&
=&\lim\limits_{n\to\infty}\frac{\deg f_n\mu(E)-\mu (\mathcal{S}(E)|_{\tilde{X}_n})}{\deg f_n}\\
&=& \lim\limits_{n\to\infty}~[\mu(E)-\mu(E)+\frac{kn^{1/2}H\cdot H}{n}]\\
&=&\lim\limits_{n\to\infty}~ \frac{kn^{1/2}H\cdot H}{n}\\
&=&0.\\
\end{eqnarray*}
\end{proof}

\begin{proof}[Theorem~\protect{\ref{main3}}][]

 Now we complete the proof of the only if direction of theorem. Suppose $E$ is strongly semistable. Then for any finite morphism $f:\tilde{X} \to X$, $f^*(E)$ is semistable. Hence for all $k$, if  $W$ is a  subsheaf of $f^*E$, 
 $\mu(W)\leq \mu(f^*E)$. Therefore $\nu_k(E)\leq \mu(E)$ for all $k$. Now  as before one can construct $\tilde{X_n}$ and by Proposition~\ref{key-proposition3}, the theorem follows.
\end{proof}

\begin{remark}
One might hope to get a similar result of Theorem~\ref{main}, without the assumption $c_1(E)=H$, without even the lifting assumptions on the surface and the bundle. But at present we have no idea how to avoid Kodaira vanishing theorem. 
\end{remark}

However Theorem~\ref{main3} has the following corollary.

\begin{cor}
\label{main}
 Let  $E$ be a strongly semistable vector bundle  of rank $r\geq 2$  with $\Delta(E)=0$ on a smooth polarized surface $(X,H)$ such that $X, E, H$ admit  liftings $\overline{X}$, $\overline{E}$, $\overline{H}$ respectively to $W_2(\mathbb{K})$. 
  Suppose $E$ is also strongly semistable with respect to $c_1(E)+ mH$ for all $m\gg 0$.  Then $\nu_k(E)=\mu(E)$ for all $k$, where slope is taken with respect to $c_1(E)+m H$ for some large $m$. 
\end{cor}
Before going to the proof, we first prove a useful lemma.
\begin{lemma}
\label{lemma1}
 Let  $E$ be a vector bundle on a smooth polarized surface $X$.  Let $\mathcal{L}$ be a line bundle on $X$. Then for any $k$,
 $$\nu_k(E)=\mu(E)\textrm{ if and only if } \nu_k(E\otimes \mathcal{L})=\mu(E\otimes \mathcal{L})$$
\end{lemma}

\begin{proof}
Note that if there exists a  sequence of finite coverings $f_n:X_n\to X$ and  subbundles $F_n$ of $f_n^*E$ of rank $k$, then  for each $n$,
$F_n\otimes f^*_n\mathcal{L}$ is also a subbundle of $f_n^*(E\otimes \mathcal{L})$ of rank $k$.  Similarly if there exists a  sequence of finite coverings $f_n:X_n\to X$ and  subbundles $G_n$ of 
$f_n^*(E\otimes \mathcal{L})$ of rank $k$, then  for each $n$,
$G_n\otimes f^*_n\mathcal{L}^{-1}$ is also a subbundle of $f_n^*E$.
We also have
 $$ \lim\limits_{n\to\infty}\frac{\mu (F_n)}{\deg f_n}+\deg \mathcal{L}=
 \lim\limits_{n\to\infty}\frac{\mu( F_n\otimes f^*_n\mathcal{L})}{\deg f_n}.$$
Thus $$\nu_k(E\otimes \mathcal{L})=\nu_k(E)+\deg \mathcal{L}.$$
Hence the lemma follows.
 \end{proof}
 
 \begin{proof}[Corollary~\protect{\ref{main}}][]
  First note that $E$ is strongly semistable if and only if  for any line bundle  $\mathcal{L}$, $E\otimes \mathcal{L}$ is so and we also have that $\Delta (E)=\Delta(E\otimes\mathcal{L})$. Also by Lemma~\ref{lemma1}, $\nu_1(E)=\mu(E)$ if and only if  $\nu_1(E\otimes \mathcal{L})=\mu(E\otimes \mathcal{L})$.  Also $c_1(E\otimes nH)=c_1(E)+rnc_1(H)$. Hence  the corollary follows from Theorem~\ref{main3}.
 \end{proof}
We conclude this section with the following remarks where we give criterion, when the hypothesis ``$E$ is strongly semistable with respect to $c_1(E)+ mH$ for all $m\gg 0$"  of 
Corollary ~\ref{main}  holds.
\begin{remark}

\begin{enumerate}
 \item  When $E$ is a vector bundle with $c_1(E)=c_2(E)=0$, then it satisfies all the hypothesis of Theorem~\ref{main}. Hence in this case given an ample line bundle $H$, $E$ is strongly semistable  with respect to $H$ if and only if $\nu_k(E)=\mu(E)$ for all $k$.
 \item If $E$ is a strongly semistable with respect to $c_1(E)$, then $E$ is strongly semistable  with respect to $c_1(E)+mH$ for all $m$, as 
$\mu_{c_1(E)+mH}(\_)=\mu_{c_1(E)}(\_)+\mu_{mH}(\_)$. 

Next suppose that $E$ is not strongly semistable with respect to $c_1(E)$
Consider $\lim_{n\to\infty}\frac{\mu_H({F_X^{n}}^*E)-\mu_H(W_n)}{p^n}$
where $W_n$ denote a maximal subsheaf of ${F_X^{n}}^* E$. Suppose the limit is nonzero say $\delta>0$. Since $E$ is not strongly semistable with respect to $c_1(E)$, then there exists $n_0$ such that ${F_X^{n}}^*$ is not semistable for all $n\geq n_0$. 
By \cite{Langer}, it is known that there exists $n_0+k$ such that if $V$ is the maximal destabilizing subsheaf ${F_X^{n_0+k}}^*(E)$ then $F_X^*(V)$ is the
maximal destabilizing subsheaf of ${F_X^{n_0+k+1}}^*(E)$. Let $V_{n_0},\ldots,V_{n_K}$ are  maximal destabilizing subsheaves of ${F_X^{n_0}}^*(E),\ldots, {F_X^{n_0+k}}^*(E)$. Let 
${\epsilon}_i=\frac{\mu_H(V_{n_0+i})-\mu_H({F_X^{n_0+i}}^*(E))}{p^{n_0+i}}$. Now choose $m$ such that $m\delta>{\epsilon}_i$ for all $i$.
Let $V$ be a subsheaf of ${F_X^n}^*(E)$, then 
\begin{eqnarray*}
\frac{
\mu_{c_1(E)+mH}({F_X^n}^*(E))-\mu_{c_1(E)+mH}(V)}{p^n}&
=&
\frac{\mu_{c_1(E)}({F_X^n}^*(E))-\mu_{c_1(E)}(V)+\mu_{mH}({F_X^n}^*(E))-\mu_{mH}(V)}{p^n}\\
&\geq& -\max\{{\epsilon}_i: i\}+m\delta\\
&>&0.\\
\end{eqnarray*}
Hence whenever  $\delta>0$, $E$ is strongly semistable with respect to $c_1(E)+mH$ for all $m\gg 0$.
\end{enumerate}

%\textbf{ If $E$ is strongly semistable w.r.t $H$, then there exists 
%$m\gg 0$ such that $E$ is strongly semistable w.r.t $c_1(E)+mH$}\\

\end{remark}
 
 %%%%%%%%%%%%%%%%%%%%%%%%%%%%%%%%%%%%%%%%%%%%%%%%%%%%%%
%\textbf{Proof of finiteness:}

%Here is a proof of flatness of general hyperplane sections of a flat family. 

%Let $X\to Y$ be a flat family obtained by a morphism of smooth varieties. Let $X\subset P^N$  be an embedding. For each $y\in Y$, let $X_y$ denote the fibre. Then the dimension of  $\{H\in \check{P}^N\}$ of hyperplanes containing a fibre is at most  $ N - \dim X_y -1 $. Hence  the total dimension of all hyperplanes that may contain some fibre is  $ N - \dim X_y -1 + \dim Y$. Hence this will never all hyperplanes if  $\dim Y \leq  \dim X_y$. Hence we can cut down until  $\dim X_y \leq 1 $ if $ \dim Y =2 $, as we are in that case. Now again even if the fibre is a curve that enequality still works unless this fibre is a line in $P^N$, which probably we can assume (by the choice of very ample ample line bundle). Then again we can get to a finite map. 

 \bibliographystyle{plain}
\bibliography{strong-restriction.bib}
\end{document}